\documentclass[11pt]{article}
\usepackage{amssymb,amsfonts,amsmath,amsthm,subfigure}
\usepackage{authblk}
\usepackage{color}
\usepackage{enumitem}
\usepackage{soul}
\usepackage[normalem]{ulem}
\usepackage{adjustbox,lipsum}
\usepackage{setspace}
\usepackage{comment}
\usepackage[colorlinks]{hyperref} 
\usepackage{txfonts}
\usepackage{graphicx}
\usepackage{nicefrac}

\usepackage{scalerel,stackengine}
\newcommand\reallywidehat[1]{%
\savestack{\tmpbox}{\stretchto{%
  \scaleto{%
    \scalerel*[\widthof{\ensuremath{#1}}]{\kern.1pt\mathchar"0362\kern.1pt}%
    {\rule{0ex}{\textheight}}
  }{\textheight}%
}{2.4ex}}%
\stackon[-6.9pt]{#1}{\tmpbox}%
}
\numberwithin{equation}{section} 

\newtheorem{thm}{Theorem}[section]
\newtheorem{defn}[thm]{Definition}

\newtheorem{lemma}[thm]{Lemma}

\newtheorem{rmrk}[thm]{Remark}

\newcommand{\R}{\mathbb{R}}

\newtheorem{prop}[thm]{Proposition}

\newcommand{\C}{\mathbb{C}}
\newcommand{\ba}{\begin{array}}
\newcommand{\ea}{\end{array}}

\newcommand{\bthm}{\begin{thm}}
\newcommand{\ethm}{\end{thm}}
\newcommand{\bstp}{\begin{stp}}
\newcommand{\estp}{\end{stp}}
\newcommand{\blemma}{\begin{lemma}}
\newcommand{\elemma}{\end{lemma}}
\newcommand{\bprop}{\begin{prop}}
\newcommand{\eprop}{\end{prop}}
\newcommand{\bpf}{\begin{pf}}
\newcommand{\epf}{\end{pf}}
\newcommand{\bdefn}{\begin{defn}}

\newcommand{\edefn}{\end{defn}}
\newcommand{\brk}{\begin{rmrk}}
\newcommand{\erk}{\end{rmrk}}
\newcommand{\bcrl}{\begin{crl}}
\newcommand{\ecrl}{\end{crl}}

\newcommand{\beqn}{\begin{equation}}
\newcommand{\eeqn}{\end{equation}}

\renewcommand{\leq}{\leqslant}

\newcommand{\beq}{\begin{equation}}
\newcommand{\eeq}{\end{equation}}
\newcommand{\bea}{\begin{eqnarray}}

\newcommand{\eea}{\end{eqnarray}}

\newcommand{\dive}{\mathrm{div}\,}

\newcommand{\bbm}{\mathbf{m}}
\newcommand{\bbu}{\mathbf{u}}

\title{Bounds on the energy of a soft cubic ferromagnet with large magnetostriction}

\begin{document}
\renewcommand\Authfont{\small}
\renewcommand\Affilfont{\itshape\footnotesize}

\author[1]{Raghavendra Venkatraman \footnote{rvenkatr@andrew.cmu.edu}}
\author[2]{Vivekanand Dabade\footnote{vivekanand.dabade@polytechnique.edu}}
\author[3]{Richard D. James\footnote{james@umn.edu}}
\affil[1]{Department of Mathematical Sciences, Carnegie Mellon University, Pittsburgh PA.}
\affil[2]{LMS, \'Ecole Polytechnique, CNRS, Universit\'e Paris-Saclay, Palaiseau 91128, France}
\affil[3]{Aerospace Engineering and Mechanics Department, University of Minnesota, Minneapolis, MN.}
\maketitle
\begin{center}
    Dedicated to Peter Sternberg on the occasion of his sixtieth birthday, with respect and admiration. 
\end{center}

\noindent {\bf Abstract:} We complete the analysis initiated in \cite{JNLS2018} on the micromagnetics of cubic ferromagnets in which the role of magnetostriction is significant. We prove ansatz-free lower bounds for the scaling of the total micromagnetic energy including magnetostriction contribution, for a two-dimensional sample. This corresponds to the micromagnetic energy-per-unit-length of an infinitely thick sample. A consequence of our analysis is an explanation of the multi-scale zig-zag Landau state patterns recently reported in single crystal Galfenol disks from an energetic viewpoint. Our proofs use a number of well-developed techniques in energy-driven pattern formation.


\section{Introduction and setup of the problem} We are interested in deriving optimal energy scaling laws for a ferromagnetic sample with cubic anisotropy. Important examples of cubic ferromagnets include Iron \cite{lifshitz1945magnetic}, Permalloy \cite{desimone2001two}, Tefenol-D \cite{desimone2002constrained}, and Galfenol \cite{chopra2015non}. These ferromagnets, when magnetized, undergo spontaneous elastic deformation; this is known as magnetostriction. Iron and Permalloy are low magnetostrictive materials, whereas Terfenol-D and Galfenol are large magnetostrictive materials. Materials with large magnetostriction exhibit a fascinating interplay of elasticity and magnetism. Inspired by recent experiments on Galfenol reported in \cite{chopra2015non}, we initiated a variational study of cubic ferromagnets with magnetostriction in \cite{JNLS2018}. In \cite{JNLS2018}, we first analyzed Young measures arising as limits of minimizing sequences for the so-called no-exchange relaxation and applied this analysis to derive macroscopic properties of Galfenol. Restoring the exchange energy term, defined below, we then derived rigorous upper bounds for the scaling of the optimal energy for the full micromagnetic energy functional in the presence of magnetostriction. Our upper bounds required fairly complex multi-scale constructions inspired by the micrographs in \cite{chopra2015non}. The goal of the present paper is to supplement this upper bound with an ansatz-free lower bound, within a two dimensional setting that is motivated by the geometry of the sample in \cite{chopra2015non}. This lower bound demonstrates that within the parameter regime of Galfenol, one can not do energetically better than our constructions from \cite{JNLS2018}. 
\begin{figure}
\begin{center}
\includegraphics[scale=0.75]{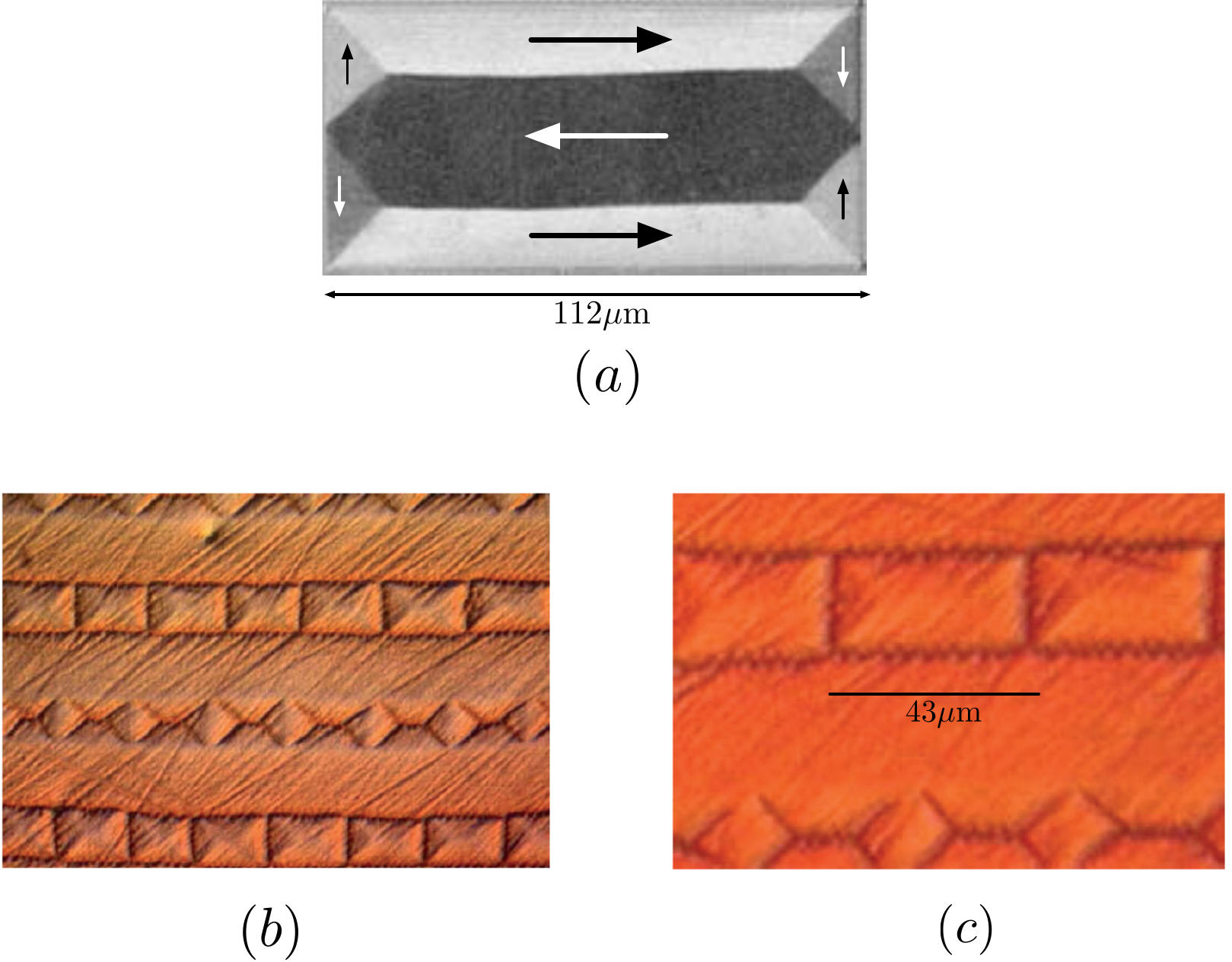}
\caption{Experimental Micrographs. (a) Normal Landau state seen in Permalloy, \cite{hubert2008magnetic}.  (b) and (c) Zig-zag Lanadau state seen in Galfenol, \cite{chopra2015non}.}
\label{fig: micrographs}
\end{center}
\end{figure}

\noindent Towards describing the functional that is at the core of our paper, we first set some notation. We let $G \subset \R^2$ denote the unit cube $\left(-\frac{1}{2},\frac{1}{2}\right)^2,$ We define the functions $\varphi:\R^2 \to \R$ and $\epsilon_0: \R^2 \to \R^{2\times 2}$ by the formulas 
\begin{align}
    \varphi(z) = \varphi(z_1, z_2) &:= \left(z_1^2 - z_2^2 \right)^2, \\
    \epsilon_0(z) =  z \otimes z {\, - \,} \frac{1}{2}I_2 &= \begin{pmatrix}
    z_1^2 & z_1 z_2\\
    z_2z_1 & z_2^2 
    \end{pmatrix} {\, - \,} \frac{1}{2}\begin{pmatrix}
    1 & 0 \\0 & 1
    \end{pmatrix},
\end{align}
{ where $z \otimes z$ denotes the tensor product of $z$ with itself, so that the matrix $(z\otimes z)_{ij} = z_i z_j, $ as indicated above}.
Finally, for a function $u \in H^1(G;\R^2),$ we define 
\begin{align}
    \label{symgrad}
    \epsilon(u) := \frac{\nabla u + (\nabla u)^T}{2}.
\end{align}

\noindent Let $v \in H^1(G;\R^2)$ and let $\tilde{v}$ denote the extension of $v$ to $\R^2$ by zero outside of $G.$ For a fixed positive number $\mu$ we consider the family of (\textit{fully non-dimensionalized}) variational problems indexed by $\eta > 0$ given by
\begin{align} \label{fullenergy}
    &\mathcal{F}_\eta (v;G) = \mu \eta \int_{G}  |\nabla v|^2 \,d{\bf{x}}+ \frac{\mu }{\eta}\int_G  \left((|v|^2 - 1)^2 + \varphi(v) \right) \,d{\bf{x}} +  \|\dive \tilde{v}\|_{H^{-1}(\R^2)}^2 \\ \notag & \quad + \inf_{\bbu \in H^1(G;\R^2)}\int_G \|\epsilon(\bbu) - \epsilon_0(v)\|^2 \,d{\bf{x}} 
\end{align}
 The motivation for this scaling and the derivation of this model will be made clear in Sec. \ref{sec:physics}; for now, let us simply remark that in this scaling, the energies $\mathcal{F}_\eta (v)$ are bounded as $\eta \to 0.$ As we will explain subsequently, the $\eta \to 0$ asymptotics of the minimum energies \eqref{fullenergy} are captured by the functional 
\begin{align}
\mathcal{F}_0(\bbm) := \mu \int_G |\nabla \bbm| + \inf_{\bbu \in H^1(G;\R^2)} \int_G \|\epsilon(\bbu) - \epsilon_0(\bbm)\|^2 \,d {\bf x} + \|\dive \, \tilde{\bbm}\|_{H^{-1}(\R^2)}^2 
\end{align}
among competitors $\bbm = (m_1, m_2)$ such that $\bbm \in \mathcal{M}$ defined by
\begin{align}
\label{d.mdom}
\mathcal{M}:= \left\{ \bbm \in BV(G;\R^2) : \bbm({\bf{x}}) \in \mathbb{K}:= \Big\{ \big( \pm \frac{1}{\sqrt{2}}, \pm \frac{1}{\sqrt{2}}  \big)\Big\} \mbox{ at almost every } {\bf{x}} \in G.\right\}.
\end{align}
For any $\bbm \in \mathcal{M}$ we denote by $\tilde{\bbm}$ its trivial extension outside $G$. We also introduce the set 
\begin{equation}\label{d.mdom0}
\begin{aligned}
\mathcal{M}_0 := \left\{ \bbm \in \mathcal{M}: \int_{-1/2}^{1/2} m_1 m_2(x,y)\,dy = 0 \quad \mbox{ for a.e. } x \in (-1/2,1/2) \mbox{ and } \right. \\ \left. \int_{-1/2}^{1/2} m_1 m_2(x,y)\,dx = 0 \quad \mbox{ for a.e. } y \in (-1/2,1/2)  \right\}. 
\end{aligned}
\end{equation}

\noindent 
Our main theorem is 
\bthm \label{mt}
There exists universal constants $0 < c_1 \leqslant 1,$ and { $c_2 > 0,$}  such that the following holds: for any $\mu \in (0, c_1),$ and { any sequence $\{v_\eta\}_{\eta > 0}, $ such that $\sup_{\eta > 0} \mathcal{F}_\eta(v_\eta) < \infty,$  
\begin{align} \label{e.upbnd}
\limsup_{\eta \to 0} \mathcal{F}_\eta (v_\eta)  \lesssim c_2 \mu^{2/3}.
\end{align}
}
If in addition, $\bbm \in \mathcal{M}_0,$ then we also have a matching ansatz free lower bound:
\begin{align} \label{e.lbnd}
\mathcal{F}_0(\bbm) \gtrsim \frac{1}{c_2} \mu^{2/3}.
\end{align}

\ethm

\noindent The proof of the upper bound inequality is essentially contained in \cite{JNLS2018}, and is recalled briefly in Section \ref{sec:upbnd}. 
The proof of the lower bound inequality is the content of Section \ref{sec:lbnd}. We conjecture that the lower bound in \eqref{e.lbnd} holds for any $\bbm \in \mathcal{M}$; we will discuss the obstructions faced in Section \ref{sec:lbnd}. The rest of this introduction is devoted to deriving the energy \eqref{fullenergy} from the micromagnetic functional.

\subsection{Derivation of the energy \eqref{fullenergy} from micromagnetics} \label{sec:physics}
\subsubsection*{Geometry and motivation for the two-dimensional reduction: } The geometry of the sample we have in mind is cylindrical, with axis along the $z-$axis. The characteristic dimension $L$ in the $x-y$ plane is significantly smaller than its thickness along the $z-$axis. This geometry is motivated by experimental values of $L\sim 10^{-5}m$ and sample thickness along the $z-$axis $\sim 10^{-3}m$; see Extended Data Figure 4 in \cite{chopra2015non}. It permits us to work with a two-dimensional energy, that we think of as the energy per unit length of an infinitely long sample; we however do not attempt to derive this energy from the full three-dimensional model via a rigorous limiting procedure. The two-dimensional nature of our model is, however, crucial to our analysis of the magnetostriction and the magnetostatic energies. Indeed, the analysis of the magnetostriction energy relies on the Fourier analysis of a certain nonlinear function of the magnetization: this is made tractable by the nonconvex constraint that the magnetization takes values in {the set $\{(\pm \frac{1}{\sqrt{2}},\pm \frac{1}{\sqrt{2}})\} $}, yielding (somewhat surprising) cancellations. We point out that the micrographs for Galfenol, which were the original motivation of our project, have essentially in-plane magnetization. Furthermore, our two-dimensional constructions in \cite{JNLS2018} accurately predict the (macroscopic) average strain as measured in experiments on Galfenol. 
\subsubsection*{Setup from micromagnetics}
\noindent Let $\Omega \subset \R^2$ denote an open bounded domain that represents the cross-section of the ferromagnetic sample. Within the variational theory of micromagnetics, the magnetization of the sample is described by a vector field $\bbm: \Omega \to \R^3$ that satisfies $|\bbm| = 1$ almost everywhere in $\Omega.$ The magnetization $\bbm$ is extended by zero outside of $\Omega.$ With an eye of working within a two-dimensional theory, we limit ourselves to competitors of the form $\bbm(x,y) = (m_1(x,y), m_2(x,y), 0).$ Our starting point towards formally deriving \eqref{fullenergy} is the full micromagnetic energy including magnetostriction, in the absence of an external applied magnetic field:

\begin{align} \label{energy1}
{\mathbf{\mathbb{F}}} (\mathbf{m})= 
            & \underbrace{{{A}}{\int_\Omega |{\nabla \mathbf{m}}|^2}\,d{\bf{x}}}_{\text{{exchange energy}}}+
              \underbrace{{\textit K_{a}}{\int_\Omega {\varphi(\mathbf{m})\,d{\bf{x}}}}}_{\text{anisotropy energy}}+
              \underbrace{c_{44} \lambda_{111}^2 \tilde{e}_\text{mag}(\bbm)}_{\text{magnetostriction energy}}+\underbrace{{\textit K_{d}}{\int_{\mathbb{R}^2} | \mathbf{{h}_\text{m}}|^{2}\,d{\bf{x}}}}_{\text{magnetostatic energy}}.   \\ \label{magstric1}
\text { where }&{\mathbf{\tilde e}}_{mag}(\bbm) = \inf_{\mathbf \mbox{ {{\mathbf{u} \in H^1({\Omega};\R^2)}}}} \int_{\Omega} \big( \mathbf{\tilde E(u)} - \mathbf{\tilde E_0}(\bbm) \big) \cdot \mathbb{\tilde C}\big( \mathbf{\tilde E(u)} - \mathbf{\tilde E_0}(\bbm) \big) \,d\mathbf{x}. 
\end{align} 

\noindent Here, $A, K_a, c_{44}, \mathbb{C}, K_d, c_{44}\lambda_{111}^2 $ are all material parameters that we describe below. The magnetostriction energy defined in \eqref{magstric1} corresponds to the least linear elastic energy associated to a preferred non-dimensional strain tensor $\tilde{\mathbf{E}}_0(\bbm).$ The last term in the energy \eqref{energy1} is the magnetostatic energy associated to a magnetization $\bbm:$ it is derived from Maxwell's equations, and in short, penalizes the divergence of the field $\bbm$ in a negative Sobolev norm. We will explain both these energies in greater detail in the paragraphs to come. We point out that in our formulation above, the total micromagnetic energy ${\mathbf{\mathbb{F}}} (\mathbf{m})$ represents the three-dimensional energy per unit length along the $z$-direction and has dimensions [energy/length]. 
\subsubsection*{Exchange and magnetocrystalline anisotropy energies}
\noindent The \textit{exchange constant} is denoted by $A$ and typically satisfies $0 < A \ll 1$. In the literature on energy-driven pattern formation, it is also common (see \cite{choksi1998bounds,choksi1999domain,JNLS2018}) to use the so-called sharp interface functional, in which the exchange energy is measured by the BV semi-norm of the magnetization $\bbm$ as opposed to the Dirichlet energy as in \eqref{energy1}. Thus, in these studies, one might see an expression of the form 
\begin{align}
    \label{sharpexchange}
    \mu \int_\Omega |\nabla \bbm|,
\end{align}
where $\mu > 0$ is the \textit{wall cost} per unit length. Before discussing how the sharp interface and diffuse energies are related, we discuss the magnetocrystalline anisotropy energy. 

\noindent The magnetocrystalline anisotropy energy, or simply anisotropy, sets certain crystallographic directions, referred to as the \textit{easy axes}, energetically preferred for the magnetization $\bbm$. The \textit{anisotropy energy} density $(K_a \varphi(\bbm))$ is determined by the anisotropy energy coefficient ${\textit K_{a}} > 0$ \footnote{Our choice of signs here is a bit different from convention: the materials that are of interest in this paper are ``negative anisotropy materials'', with $K_a < 0$ and correspondingly $\varphi$ is defined by the negative of Eq. \eqref{anisotropy}, nevertheless rendering the product $K_a \varphi$ nonnegative. } and $\varphi(\mathbf{m})$ given by
\begin{align} \label{anisotropy}
\varphi(\mathbf{m}) = \bigg(\frac{1}{4} -  m_{1}^2 m_{2}^2\bigg) = \frac{1}{4}(m_1^2 - m_2^2)^2.
\end{align}
The wells of the anisotropy energy are referred to as the \textit{easy axes} of the sample, and in our case are given by $\pm \bbm_1, \pm \bbm_2,$ where ${\bf{m}_{1}}=  \big(\nicefrac{1}{\sqrt{2}},\nicefrac{1}{\sqrt{2}},0\big)$ and ${\bf{m}_{2}}= \big(\nicefrac{1}{\sqrt{2}},-\nicefrac{1}{\sqrt{2}},0\big)$. { Note that the anisotropy energy density given in (\ref{anisotropy}) suppresses the out of plane magnetization.}

\noindent How are the sharp-interface version of the exchange energy, \eqref{sharpexchange} and the diffuse counterpart in \eqref{energy1} related? To answer this question, it is helpful to record the dimensions of the various quantities in question. Since our functional $\mathbb{F}$ from \eqref{energy1} has dimensions of energy per unit length (in the $z-$direction),  one has
\begin{align} \label{dimensions}
  [A] = \frac{[\mbox{Energy}]}{[\mbox{Length}]}, \hspace{.5cm} [K_a] = \frac{[\mbox{Energy}]}{[\mbox{Length}]^3} \hspace{.5cm} [\mu] = \frac{[\mbox{Energy}]}{[\mbox{Length}]^2}.
\end{align}
For sufficiently large values of the anisotropy constant $K_a$, the magnetization $\bbm$ stays close to the easy axes of the sample, thus being essentially piecewise constant and forming \textit{magnetic domains}. Different domains are separated by thin transition layers. Competition between the diffuse exchange energy $A \int |\nabla \bbm|^2 $ and the anisotropy energy $\int K_a\varphi(\bbm)$ sets a \textit{surface tension} $\mu$ that effectively penalizes the \textit{surface area of the transition layer} $\int |\nabla \bbm|.$ The width of a transition layer must necessarily be smaller than the characteristic length $L$, which yields 
\begin{align} \label{thintransition}
    \sqrt{\frac{A}{K_a}} < L.
\end{align}
Under these circumstances, one can show that the surface tension is related to the exchange constant $A$ by 
\begin{align} \label{Aandgamma}
    \mu^2 \sim A K_a.
\end{align}
\noindent From the point of view of optimal energy scaling laws, these two formulations are asymptotically equivalent due to the Modica-Mortola inequality, see \cite[Section 6.8]{desimone2006recent}. The sharp-interface formulation has certain advantages: it permits one to focus attention on the domain morphology without having to simultaneously resolve the internal structure of walls. It is the sharp-interface formulation that we used in \cite{JNLS2018}, because this simplified our computations concerning the upper bound. The rigorous connection between the sharp interface and diffuse formulations is conveniently done using $\Gamma-$convergence; see \cite{sternberg86}, also \cite[Section 6.8]{desimone2006recent}. The diffuse formulation naturally has an extra small length-scale ${\eta} > 0$ corresponding to the diffuse wall thickness, as compared to the sharp interface limit. The ${\eta} \to 0$ limiting procedure yielding the sharp interface limit can then be made precise in the parameter regime $A \sim \mu  {\eta}, K_a \sim \frac{\mu}{{\eta}},$ consistent with \eqref{Aandgamma}. \par   
While the magnetization $\bbm$ is $S^1-$valued and the diffuse exchange energy which is present in the full micromagnetic energy \eqref{energy1} penalizes the $H^1-$seminorm of $\bbm$, it is well known \cite{bbh} that $S^1-$valued vector fields in the plane having vortices have infinite $H^1$-seminorm. However, even the normal Landau state, refer to Figure \ref{fig: micrographs} (a) has vortices, at each triple junction. \par 
\noindent A convenient ``remedy'' to this issue is to relax the ``hard'' constraint $|\bbm| = 1,$ and replace $\bbm$ by a vector field $v : \Omega \to \R^2$ along with a penalty term in the energy which forces $v$ to be nearly $S^1$-valued; see again \cite{bbh}. This corresponds to the term Ginzburg-Landau term $\frac{\mu}{\eta}\int_G (|v|^2 - 1)^2 \,d{\bf{x}}$ in the energy \eqref{fullenergy}, where $\eta$ is a non-dimensional version of { the wall thickness}, as will be explained below. In this scaling, the cost of a vortex is $\eta |\log \eta|$ which vanishes in the $\eta \to 0^+$ limit considered in Theorem \ref{mt}. \par 
\noindent While the Ginzburg-Landau penalty might seem like a mathematical artefact, it can be physically thought of as penalizing out-of-plane magnetization, and the walls correspondingly as Bloch walls. Since we wish to work with a two-dimensional theory, we do not pursue this interpretation; the reader might wish to see \cite{fengbo} for instance. 
\subsubsection*{Magnetostriction energy.}
\noindent We next turn to the \textit{magnetostriction energy}, the third term in \eqref{energy1}. Our reference for modeling this energy is \cite{james1998magnetostriction} which relies on linear elasticity. For notational consistency with \cite{james1998magnetostriction}, and for the convenience of the reader, we briefly describe full three-dimensional magnetostriction. Subsequently, we describe our two-dimensional reduction. The preferred strain associated to a magnetization $\bbm = (m_1, m_2, m_3): \Omega \to \mathbb{S}^2$ is given by 
\begin{align}
{\bf{E}}_{0}({\bf{m}}) = \frac{3}{2} \bigg(  \lambda_{100} ({\bf{m}} \otimes {\bf{m}} - \frac{1}{3} {\bf{I}})
+ (\lambda_{111} - \lambda_{100}) \sum_{i \ne j} {m}_i {m}_j   {\bf{e}}_i \otimes {\bf{e}}_j \bigg),\label{Ehat1}
\end{align}
where the vectors \{${\bf{e}}_1, {\bf{e}}_2, {\bf{e}}_3$\} in \eqref{Ehat1} refer to an orthonormal basis parallel to the cubic axes. The constants $\lambda_{100}$ and $\lambda_{111}$ are referred to as the magnetostriction constants of the cubic material.

The elastic energy associated to a magnetization $\bbm$ and a displacement $\bbu \in H^1(\Omega;\R^3)$ is given by 
\begin{align*}
    \frac{1}{2}\int_\Omega \left(\mathbf{E}(\mathbf{u)} - \mathbf{E}_0(\bbm) \right): \mathbb{C} \left( \mathbf{E}(\mathbf{u)} - \mathbf{E}_0(\bbm) \right)\,d \mathbf{x}, \hspace{1cm} \mathbf{E}(\bbu) = \frac{\nabla \bbu + \nabla \bbu^T}{2}.
\end{align*}

\noindent In the above, $\mathbb{C}$ is a fourth order, positive-definite, symmetric tensor, referred to as the elastic modulus. For a cubic material such as Galfenol, the elastic modulus $\mathbb{C}$ consists of three independent components: $c_{11}, c_{12}$ and $c_{44}$.  Minimizing the elastic energy over all mechanically compatible strains, i.e. all strains $\mathbf{E}$ that arise as a symmetrized gradient of an $H^1-$displacement field $\bbu$ results in \eqref{magstric1}. For a brief discussion on the role of mechanical compatibility in our variational problem, we refer the reader to \cite{JNLS2018}. 

\noindent With this background on magnetostriction, we turn to making simplifications that result in a two-dimensional theory that we use in our analysis. First, for Galfenol, one has $c_{11} \approx c_{12} \approx c_{44} \approx 10^{11}$ $\nicefrac{N}{m^2}$, refer to \cite{zhang2005phase}. We will therefore only use one elastic constant, namely $c_{44},$ and set 
\begin{align} \label{cij}
    c_{11} = c_{12}=c_{44}.
\end{align}
Furthermore, as for the magnetostriction constants, refer to \cite{JNLS2018} and references therein, one has $\lambda_{100} \approx \lambda_{111} \approx 10^{-4}$. Consequently, we set 
\begin{align}
    \label{lambdas}
    \lambda_{100} = \lambda_{111}.
\end{align}
With these assumptions, the preferred strain simplifies to 
\begin{align}
{\bf{E}}_{0}({\bf{m}}) =  \frac{3\lambda_{111}}{2}  \bigg(({\bf{m}} \otimes {\bf{m}} - \frac{1}{3} {\bf{I}}\bigg).\label{Ehat2}
\end{align}

\noindent Second, we note that in our two-dimensional framework, since $\bbm$ is in-plane, i.e. $\bbm$ takes the form $(m_1(x,y), m_2(x,y))$ and $m_3 = 0,$ the preferred strain reduces to 
\begin{equation}
\mathbf{E_{0}(m)}= \frac{3\lambda_{111}}{2}
\begin{pmatrix}
m_1^2 - \frac{1}{3} & m_1 m_2 & 0\\
m_1 m_2 & m_2^2 - \frac{1}{3} & 0\\
0 & 0 & -\frac{1}{3}
\end{pmatrix}. \, \,
\notag
\end{equation}
Motivated by the micrographs in \cite{chopra2015non}, a more significant restriction that we make is to look at displacements of the form 
\begin{equation} \label{dispassump}
\mathbf{u}(x,y,z)= \left(u_1(x,y), u_2(x,y), \frac{-\lambda_{111}}{2} z\right)    
\end{equation}
With this choice, the actual strain is given by
\begin{equation}
{\mathbf{E}}(\mathbf{u})=
\begin{pmatrix}
\frac{\partial u_{1}}{\partial x} & \frac{1}{2}\bigg(\frac{\partial u_{1}}{\partial y}+\frac{\partial u_{2}}{\partial x}\bigg) &  0        \\
\frac{1}{2}\bigg(\frac{\partial u_{1}}{\partial y}+\frac{\partial u_{2}}{\partial x}\bigg) & \frac{\partial u_{2}}{\partial y}  &  0       \\
0 & 0  &  -\frac{\lambda_{111}}{2}       \\
\end{pmatrix}. \, \,
\notag
\end{equation}
It is thus clear that we can identify $\bbu$ with a vector in $\R^2$ of the form $\bbu(x,y) = (u_1(x,y),u_2(x,y)),$ and correspondingly identify the actual and preferred strains with their top-left $2 \times 2$ blocks, viz. 
\begin{align}
    \overline{\epsilon_0}(\bbm) &= \frac{3\lambda_{111}}{2} \left( \bbm \otimes \bbm  - \frac{1}{3} I_2\right) = \frac{3\lambda_{111}}{2}\begin{pmatrix}
    m_1^2 - \frac{1}{3} & m_1 m_2 \\
    m_1 m_2 & m_2^2 - \frac{1}{3}
    \end{pmatrix} , \\
   \overline{\epsilon}(\bbu) & = \begin{pmatrix}
    \frac{\partial u_{1}}{\partial x} & \frac{1}{2}\bigg(\frac{\partial u_{1}}{\partial y}+\frac{\partial u_{2}}{\partial x}\bigg)        \\
\frac{1}{2}\bigg(\frac{\partial u_{1}}{\partial y}+\frac{\partial u_{2}}{\partial x}\bigg) & \frac{\partial u_{2}}{\partial y} 
    \end{pmatrix}.
\end{align}

\noindent As our third simplification, we note that the constraint $m_1^2 + m_2^2 = 1$ renders the tensor $\overline{\epsilon_0}(\bbm)$ to have trace $\frac{\lambda_{111}}{2}.$ For simplicity in our estimates, it is desirable to have the preferred strain be trace-free. We therefore define 
\begin{align}
    \label{strainsfinal}
    {\overline{\overline{\epsilon_0}}}(\bbm) = \overline{\epsilon_0}(\bbm) {\, + \,} \frac{\lambda_{111}}{2}I_2, \hspace{1cm}
    {\overline{\overline{\epsilon}}}(\bbu) = \overline{\epsilon}(\bbu) {\, + \,} \frac{\lambda_{111}}{2}I_2, 
\end{align}
where $I_2$ is the identity matrix in $\R^2$. Obviously, this does not change the elastic energy associated to a magnetization $\bbm$ and a corresponding displacement $\bbu$ of the form \eqref{dispassump}.

\noindent 
Our last simplification is one of non-dimensionalization: we set
\begin{align}
    \label{strainsfinal1}
     \mathbb{\tilde C}=\frac{\mathbb{C}}{c_{44}},\hspace{0.5cm} {\epsilon_0}(\bbm)=\frac{{\overline{\overline{\epsilon_0}}}(\bbm)}{\lambda_{111}} , \hspace{0.5cm}
    {\epsilon}(\bbu)=\frac{{\overline{\overline{\epsilon}}}(\bbu)}{\lambda_{111}}.  
\end{align}

\noindent Putting together \eqref{cij}, \eqref{lambdas}, \eqref{strainsfinal} and \eqref{strainsfinal1} we find the magnetostriction energy from equation \eqref{magstric1} associated to a magnetization $\bbm$ is given by 
\begin{align}
    \label{magfinal}
    \inf_{\bbu \in H^1(\Omega;\R^2)} c_{44}\lambda_{111}^2\int_\Omega \|\epsilon(\bbu) - \epsilon_0(\bbm)\|^2 \,d{\bf{x}}
\end{align}
with $\| A \|^2$ denoting the sum of the square of the entries of the matrix $A$. We will denote the  magnetostriction energy coefficient as $c_{44}\lambda_{111}^2$.
\subsubsection*{Magnetostatic energy.}
\noindent The final term in our energy is the \textit{magnetostatic energy} and the relevant material paramater is known as magnetostatic energy coefficient $\textit K_{d}$. The magnetostatic energy penalizes the induced or stray field $\mathbf{h}_\bbm$ associated to the magnetization $\bbm.$ The induced field  $\mathbf{{h}_\text{m}}$ is obtained by solving Maxwell's equations of magnetostatics on $\R^2$,
\begin{subequations}
\begin{align}
\nabla\cdot(\mathbf{{h}_\text{m}} + \mathbf{m}) &= 0, \label{maxwella}\\
\nabla \times \mathbf{{h}_\text{m}}&=0\label{maxwellb},
\end{align}
\end{subequations}
in $H^{-1}(\R^2)$. We remind the reader that since our sample is infinitely thick in the $z-$direction, the magnetostatic energy in \eqref{energy1} is interpreted as the magnetostatic energy per unit length of the sample in the $z$-direction. It is then easily seen that 
\begin{align*}
    \int_{\R^2} |\mathbf{h}_\bbm|^2 \,d{\bf{x}} = \|\dive \bbm\|_{H^{-1}(\R^2)}^2. 
\end{align*}

\noindent 


\subsubsection*{Parameter regime and derivation of the functional \eqref{fullenergy}}
A primary motivation for our project is the fascinating two-scale microstructure in Galfenol \cite{chopra2015non}; the authors there refer to this pattern as the \textit{zig-zag Landau state}. The magnetic microstructure in Galfenol is in striking contrast to known traditional soft ferromagnets such as Permalloy, that exhibit the so-called ``\textit{normal Landau state}''; refer to Figure 1. { The normal Landau state has rectangular boundary made up of straight lines while the zig-zag Landau state  has a boundary with corrugated or zig-zag lines. We refer the reader to section \ref{sec:upbnd} for a detailed description of the magnetization in the zig-zag Landau state.} 

Our point of view in \cite{JNLS2018} and the present paper is to explain this complex microstructure as the result of the competition between magnetostriction energy, which prefers high frequency oscillations in the magnetization, and the small yet nonzero wall energy, which favors relatively few domain walls. Indeed, the magnetostrictive strains in Galfenol ($\approx 10^{-4}$) are much larger than traditional ferromagnets ($\approx 10^{-6}$). Furthermore, the large magnetostriction energy coefficient in Galfenol is comparable to the anisotropy energy coefficient, i.e. $c_{44} \lambda_{111}^2 \approx {\textit K_{a}}\approx  10^{3}$. In contrast, in Permalloy, the magnetostriction energy coefficient is much smaller than the anisotropy energy coefficient, i.e. $c_{44} \lambda_{111}^2 \approx 10^{-1} << {{\textit K_{a}}} \approx 10^{2}$. 

 In \cite{JNLS2018} we constructed an upper bound for the micromagnetic energy based on a zig-zag Landau state construction. The construction reported there was an interpretation of the micrographs from \cite{chopra2015non}. The goal of our paper is to prove a matching ansatz-free lower bound. For clarity, we work within a parameter regime of a soft ferromagnet in which magnetostriction is strongly coupled with anisotropy. In terms of physical units, we assume $0 < A \ll 1, c_{44} \lambda_{111}^2 \approx K_d \ll K_a.$ Furthermore, we suppose that the sample cross-section is given by the square $\left(-\frac{L}{2},\frac{L}{2}\right)^2.$ Rescaling the domain by the characteristic length $L$, we arrive at a functional defined on the unit square 
 \begin{align*}
     G := \left( -\frac{1}{2}, \frac{1}{2} \right)^2.
 \end{align*}
 Non-dimensionalizing the energy by dividing through by $c_{44}\lambda_{111}^2 L^2 ,$ and defining the (non-dimensional) positive numbers $\mu, \eta,$ via 
 \begin{align*}
         \mathcal{F}_\eta(v) &:= \frac{1}{c_{44} \lambda_{111}^2 L^2 } \mathbb{F}(v), \hspace{1cm}
         &\mu \eta := \frac{A}{c_{44}\lambda_{111}^2 } , \\
         \frac{\mu}{\eta} &:= \frac{K_a}{c_{44} \lambda_{111}^2}, \hspace{1cm}
 \end{align*}
 we arrive at the energy \eqref{fullenergy}. Here, $\mu$ plays the role of a non-dimensional surface tension, refer to \eqref{Aandgamma}, and $\eta$ a non-dimensional diffuse wall-thickness.

\section{Preliminaries} \label{sec:prelim}
\subsection{On the magnetostriction energy}
We recall the following version of Korn's inequality (refer to \cite{Nitsche}) that we will use to show that for any magnetization $\bbm,$ one has a displacement $\bbu$ that achieves the infimum in \eqref{magfinal}. 
\bthm \label{korn}
Let $\Omega \subset \R^n$ denote a bounded, open set with Lipschitz boundary. There exists a constant $C(n,\Omega)$ such that 
\begin{align}
    \|\nabla u\|_{H^1(\Omega)} \leqslant C \left\|\frac{\nabla u + \nabla u^T}{2} \right\|_{L^2(\Omega)} ,
\end{align}
for all $u \in H^1(\Omega;\R^n)$ such that 
\begin{enumerate}
    \item for $i \in \{1,\cdots, n\},$ we have $\int_\Omega u_i \,d{\bf{x}} = 0, $
    \item the matrix $a_{ij} := \left[\int_\Omega \nabla_i u^j \,d{\bf{x}} \right]$ is symmetric.
\end{enumerate}
\ethm
\noindent Using this theorem, concerning the variational problem in \eqref{magfinal} we prove
\bthm \label{directmethod}
Let $\bbm \in L^2(\Omega).$ Then there exists $\bbu_0 \in H^1(\Omega)$ with $\int_\Omega \bbu_0^i\,d{\bf{x}} = 0$ and $\left[ \int_\Omega \nabla_i \bbu_0^j \,d{\bf{x}}\right]_{ij}$ symmetric, such that 
\begin{align} \label{magstricrhs}
    \int_\Omega \|\epsilon(\bbu_0) - \epsilon_0(\bbm)\|^2 \,d{\bf{x}} = \inf_{\bbu \in H^1(\Omega;\R^2)} {\int_\Omega} \|\epsilon(\bbu) - \epsilon_0(\bbm)\|^2 \,d{\bf{x}}
\end{align}
\ethm
\begin{proof}
The proof is an easy application of the direct method in the Calculus of Variations, and we outline it. For ease of notation, set $V := e_0(\bbm)$ and note that $\|V\|_{L^2(\Omega)} \leqslant C.$ Let $\{\bbu_j\} \subset H^1(\Omega;\R^2)$ denote a minimizing sequence for the variational problem in \eqref{magstricrhs}. Since the energy on the right hand side of \eqref{magstricrhs} does not change upon adding constants and infinitesimal rotations, we may assume that for each $j \in \mathbb{N},$ one can
\begin{enumerate}
    \item add an appropriate constant to each $\bbu_j$ to arrange $\int_\Omega (\bbu_j)_i \,d{\bf{x}} = 0,$ for $i \in \{1, \cdots N\},$
    \item add an appropriate infinitesimal rotation $W_j \mathbf{x}$ to $\bbu_j$, with $W_j$ skew symmetric, so that for each $j \in \mathbb{N},$ we can arrange that the matrix $c_{ik}^j :=\left[\int_\Omega \nabla_i (\bbu_j)^k \,d {\bf{x}}  \right]$ is symmetric: that is $c_{ik}^j = c_{ki}^j$ for each $j \in \mathbb{N}$, and for all $i,k \in \{1,2\}.$ 
    \end{enumerate}
    These operations do not change the energy in \eqref{magstricrhs} of the functions $\bbu_j$. 
Denoting by $\mathfrak{m}$ the $\inf$ on the right hand side of \eqref{magstricrhs}, one easily obtains by Korn's inequality, refer to Theorem \ref{korn} that for all $j$ sufficiently large, 
\begin{align*}
    \| \bbu_j\|_{H^1(\Omega)}^2  \leqslant C \left( \|\epsilon(\bbu_j) - V\|^2 + 1\right) \leqslant C(\mathfrak{m} + 2). 
\end{align*}
The result follows by usual compactness and weak-lower semicontinuity theorems. 
\end{proof}
\noindent In the next lemma, we obtain a Fourier representation for the magnetostriction energy in the special case that 
\begin{align*}
    \Omega = G := \left( - \frac{1}{2}, \frac{1}{2} \right)^2.
\end{align*}
For the remainder of the paper, it is this cross-section that we will work with. 
\begin{lemma} \label{lemfourrep}
Let $V  \in L^2(G;\R^{2\times2}).$ Then 
\begin{align} \label{Fourierrep}
    \inf_{{\bf{u}} \in H^1(G;\R^2)} \int_G \|\epsilon({\bf{u}}) - V\|^2 \,d{\bf{x}} =  \sum_{{\bf{k}} \in \mathbb{Z}^2 \backslash \{0\}} \frac{1}{|{\bf{k}}|^4}\left(|{\bf{k}}|^4\|\widehat{V}({\bf{k}})\|^2 - 2 |{\bf{k}}|^2 |\widehat{V}({\bf{k}}){\bf{k}}|^2 + |{\bf{k}} \cdot \widehat{V}({\bf{k}}){\bf{k}}|^2  \right) ,
\end{align}
with 
\begin{align*}
    \widehat{V}({\bf{k}}) := \int_G V({\bf{x}}) e^{-2\pi i {\bf{k}} \cdot {\bf{x}}} \,d{\bf{x}}.
\end{align*}
\end{lemma}
\begin{proof}
Let $V$ be as in the Lemma, and let $u_0$ denote the minimizer obtained from Theorem \ref{directmethod}. We know that $u_0 \in H^1(G)$ are weak solutions of the Euler-Lagrange equations given by
\begin{align*}
    \dive(\epsilon(u_0) - V) &= 0, \hspace{1cm} {\bf{x}} \in G, \\
    \left(\epsilon(u_0) - V\right)\nu &= 0, \hspace{1cm} {\bf{x}} \in \partial G \backslash C
\end{align*}
with $C$ denoting the corners of the domain $G.$ Consider now the larger square $G^* := (-\frac{1}{2}, \frac{3}{2})\times (-\frac{1}{2}, \frac{3}{2}).$ 
We define $V^*$ on $G^*$ as follows: first, define $V^* = V$ on $G \subset G^*.$ On the square $(\frac{1}{2}, \frac{3}{2}) \times (-\frac{1}{2},\frac{1}{2})$ we define $V^*$ be performing an even reflection of $V$ in the $x-$variable about the side $\{x = \frac{1}{2}\} \cap \overline{G}.$ Finally, we define $V^*$ on the rectangle $(-\frac{1}{2},\frac{3}{2}) \times (\frac{1}{2}, \frac{3}{2})$ by an even reflection in the $y-$variable of $V^*$ defined thus far, about the line $\{y = \frac{1}{2}\} \cap G^*.$ 
\noindent We denote by $u_0^*$ the result of performing the foregoing reflection procedure to $u_0.$ It is clear, thanks to the even reflection that $u_0^* \in H^1(G^*),$ and is $G^*-$periodic. 
We now consider the variational problem 
\begin{align} \label{periodic}
    \inf_{w \in H^1_\#(G^*;\R^2)} \int_{G^*} \|\epsilon(w) - V^*\|^2 \,d{\bf{x}} ,
\end{align}
where $H^1_\#(G^*;\R^2)$ consists of $G^*-$periodic $H^1$ vector fields in $\R^2.$ We note that up to addition of constants and infinitesimal rotations, this problem has a unique minimizer. We claim that $u_0^* \in H^1_\#(G^*)$ is a minimizer to this variational problem. Indeed, by convexity, it suffices to verify the weak form of the Euler-Lagrange equations. In fact, it suffices to verify the weak form of the Euler-Lagrange equations associated to \eqref{periodic} in neighborhoods of points along $\partial G^*$ (away from the corners). To this end, we let $B = B(\mathbf{x},r)$ denote a ball centered at $\mathbf{x} \in \partial G^* \backslash C$ and radius $r < 1.$ We test against functions $\phi \in C_c^\infty(B;\R^2), $ and we write $B = B_+ \cup B_-$ with $B_+ = B \cap G^*$ and $B_- = B \backslash \overline{G^*}$ . By integration by parts, we find
\begin{align*}
    \int_B \epsilon(\phi) : \left( \epsilon(u_0^*) - V^*\right) \,d{\bf{x}} &= - \int_{B_+ } \phi \cdot \dive \left( \epsilon(u_0^*) - V^*\right) \, d {\bf{x}} - \int_{B_-} \phi \cdot \dive \left( \epsilon(u_0^*) - V^*\right) \,d{\bf{x}} \\
    &+ \int_{\partial G \cap B} \phi \cdot \left(\epsilon(u_0^*) - V^* \right)_+ \nu -  \int_{\partial G \cap B} \phi \cdot \left(\epsilon(u_0^*) - V^* \right)_- \nu \\&= 0,
\end{align*}
thanks to the Euler-Lagrange equations satisfied by $u_0$, and crucially, the natural boundary conditions. Here, subscripts $\cdot_\pm$ respectively denote the traces of the periodized quantities along $\partial G^*.$ \par \noindent 
Having shown this, the Fourier representation follows as in the proof of \cite[Lemma 4.1]{knupfer2013nucleation}. 
\end{proof}
\section{Upper bound: the results of \cite{JNLS2018} and a modification} \label{sec:upbnd}
\noindent In our previous paper \cite{JNLS2018}, the energies of laminates of the normal Landau state and of the zig-zag Landau state were compared. The zig-zag Landau state refers to the magnetization pattern reported in the experiments of Chopra and Wuttig \cite{chopra2015non}, also see Figure \ref{fig: micrographs}(b) and \ref{fig: micrographs}(c), whereas, the normal Landau state is the magnetization pattern observed in more traditional cubic materials such as Permalloy, see Figure \ref{fig: micrographs}(a).

At the level of energies, comparing the two in the parameter regimes of Galfenol shows that the zig-zag Landau state is energetically favored compared to the normal Landau state. This is striking, because the zig-zag Landau state is a significantly more complex, two scale construction, as opposed to a single-scale normal Landau state laminate. We showed in \cite{JNLS2018} that the zig-zag Landau state has a coarse microstructure in regions of mechanical compatibility of the preferred strain and a fine scale microstructure near the regions of incompatibility of the preferred strain, refer to the discussion in \cite[Section 2.1 and Lemma 4.2]{JNLS2018}. 

\begin{figure}
\begin{center}
\includegraphics[scale=0.510]{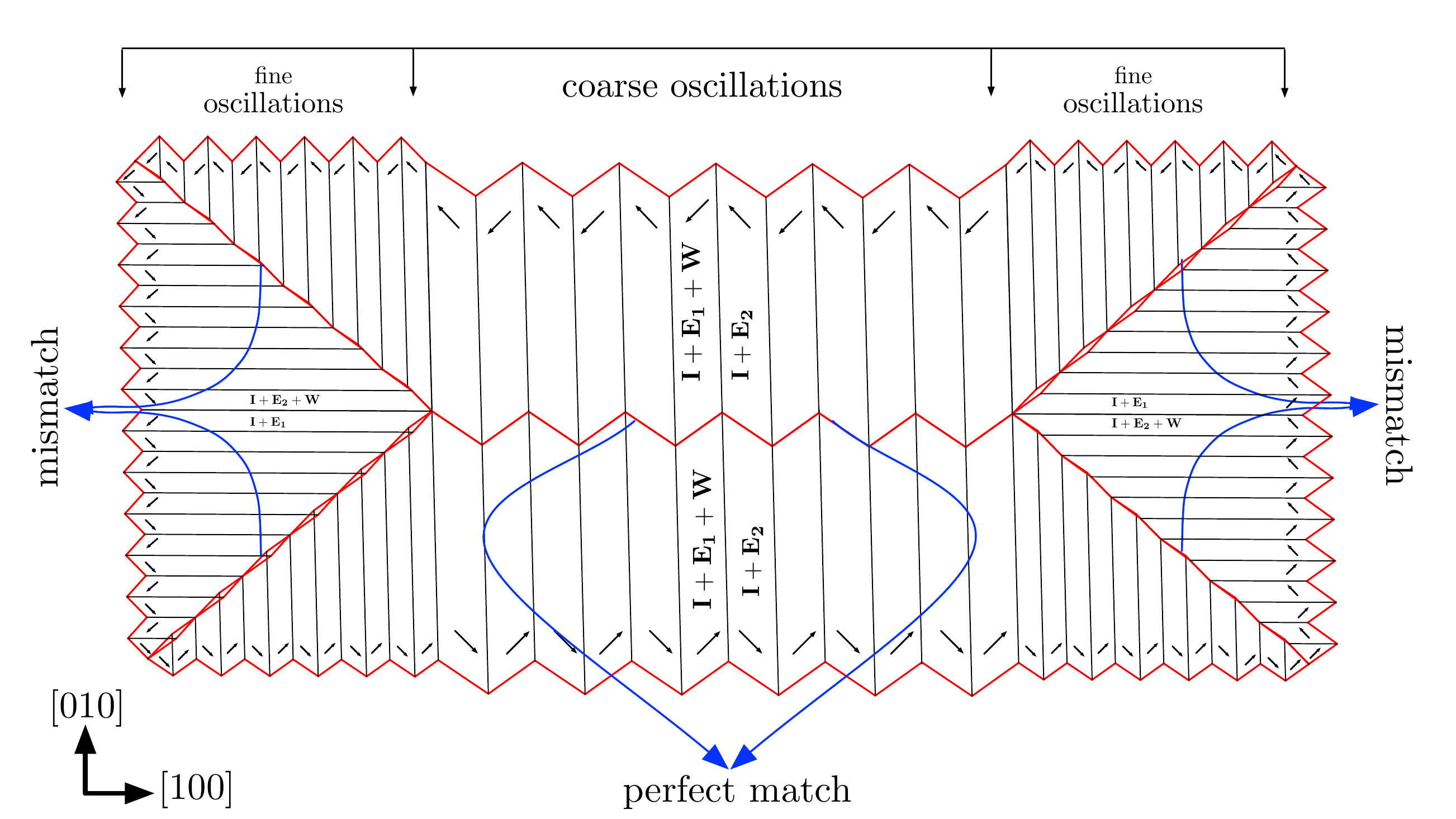}
\caption{Deformed zig-zag Landau state with no transition layer. The preferred strains: $\mathbf{E}_{1}$ and $\mathbf{E_{2}}$ and the infinitesimal rotation $\mathbf{W}$ are given in equation (56) of \cite{JNLS2018} . }
\label{fig: zig-zag_2}
\end{center}
\end{figure}

\noindent  Towards recalling this construction and presenting a different version of it, we note that the easy axes of a cubic material consists of $\left\{\left(\pm \frac{1}{\sqrt{2}}, \pm \frac{1}{\sqrt{2}}\right)\right\},$ and thus, two kinds of walls make up most of our constructions: $90^\circ$ walls, and $180^\circ$ walls. In \cite{JNLS2018}, we made a construction which was divergence free, motivated by the large $K_d$-value for Galfenol. Here, we briefly present a slight modification of that construction that is relevant for cubic ferromagnets with large and comparable magnetostriction  and magnetostatic energies and significantly larger magnetocrystalline anisotropy.

The fundamental building block of both constructions is the single zig-zag Landau state unit cell, shown in Figure \ref{fig: zig-zag_2}. Both our constructions consist in the bulk of $k \in \mathbb{N}$ single zig-zag Landau states in the sample $G$; see Figure \ref{fig: zig-zag_5}. It is easily checked that the number of $180^\circ$ walls is comparable to $k$. In regions of mechanical incompatibility, the zig-zag Landau state construction consists of a further fine-scale oscillation that predominantly makes use of {`\it{l}'} $ \, 90^\circ-$ walls. 

\noindent The difference between the constructions we presented in \cite{JNLS2018} and the modification we describe here lies in the triangular boundary domains. In the construction in \cite{JNLS2018}, these consisted of closure domains where the magnetization \textit{does not} lie along the easy axes, but is divergence free. In the modification we present in Figure \ref{fig: zig-zag_5}, the magnetization is \textit{not} divergence-free, but lies on the easy axes. We have highlighted the magnetization in four representative boundary triangles in Figure \ref{fig: zig-zag_5}.

\noindent This magnetization pattern $\mathbf{m}$ is shown in Figure \ref{fig: zig-zag_5}, where $k=2$. In this construction $\mathbf{m} \in \{\pm \bbm_1, \pm \bbm_2\} $, and so this construction has zero anisotropy energy.  

\noindent Aside from the boundary triangles described above, the modification in \ref{fig: zig-zag_5} is identical to the constructions in \cite{JNLS2018}:  each zig-zag Landau state is a second order laminate consisting of two distinct scales of oscillation frequencies, a coarse scale oscillation of frequency $k\sim\frac{L^{\frac{1}{3}} (c_{44}{\lambda^{2}_{111}} + \textit K_{d})^{\frac{1}{3}}}{\gamma^{\frac{1}{3}}}$ and a fine scale oscillation of frequency $lk\sim\frac{L^{\frac{2}{3}} (c_{44}{\lambda^{2}_{111}} + \textit K_{d})^{\frac{2}{3}}}{\gamma^{\frac{2}{3}}}$

\noindent Calculating the energies of the both constructions is identical with the exception that the present construction also has a magnetostatic contribution. We remind the reader that in \cite{JNLS2018} we worked with the sharp interface energy, which prior to non-dimensionalizing reads
\begin{align}
    \mathbb{F}_\#(\bbm) = \mu L \int_G |\nabla \bbm| + K_a L^2\int_G \varphi(\bbm) \, d {\bf{x}} + K_d L^2 \int_{\R^2}|\mathbf{h}_\bbm|^2 \,d {\bf{x}} + c_{44} \lambda_{111}^2 e_{mag}(\bbm),
\end{align}
with competitors that satisfied $\bbm \in BV(G;\R^2).$ Estimating the magnetostriction energy of this magnetization proceeds identically to \cite{JNLS2018}: for a detailed description of the magnetization, and the deformation gradients in the sample $G$ away from the boundary triangles which remain unchanged for the present construction, we refer the reader to \cite[Section 4.3]{JNLS2018}. 

\begin{figure}
\begin{center}
\includegraphics[scale=0.15]{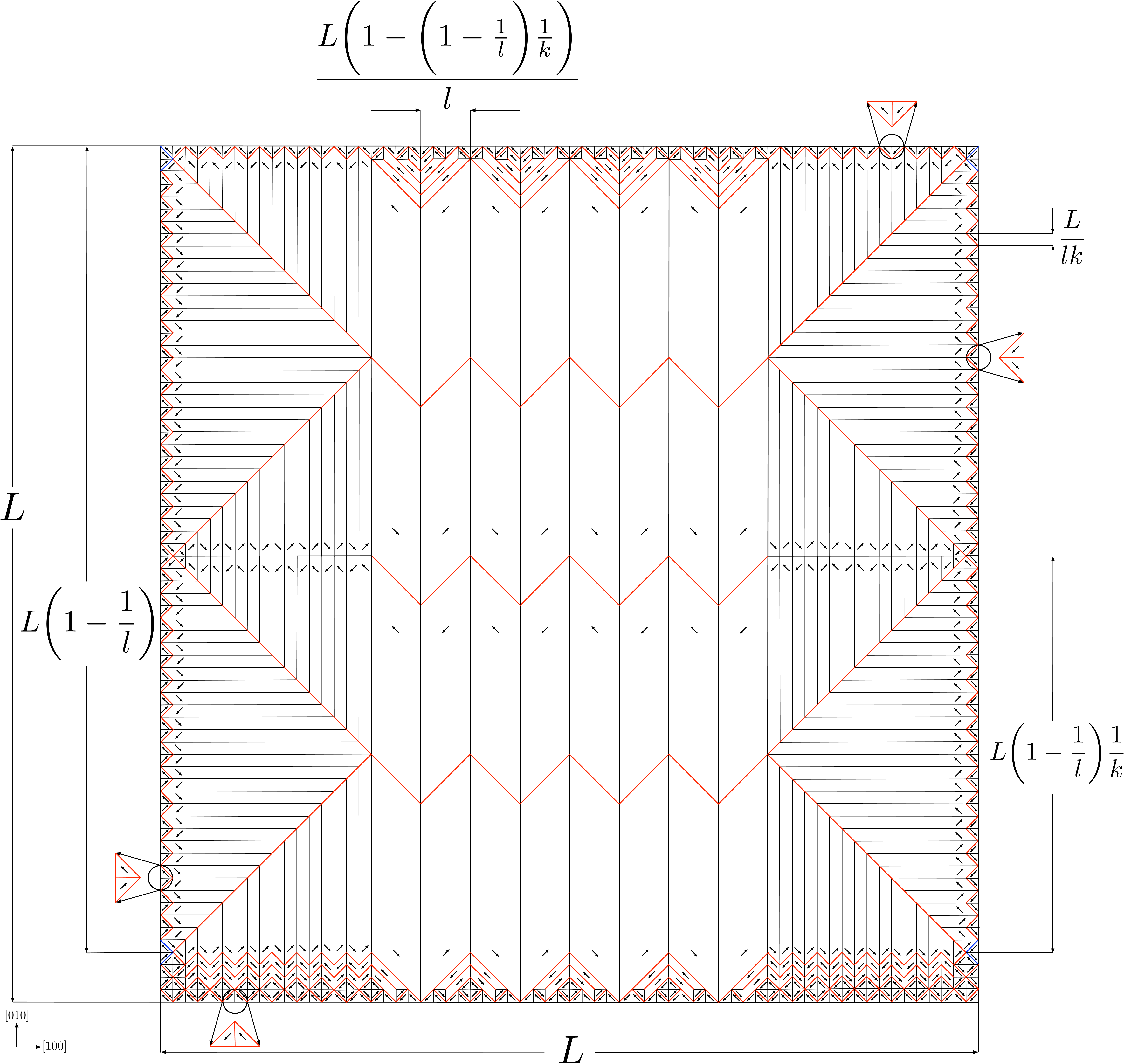}
\caption{Magnetization in $(-\frac{L}{2},\frac{L}{2})\times(-\frac{L}{2},\frac{L}{2})$ square consisting of $k  (=2)$ zig-zag Landau states for cubic ferromagnet with large and comparable magnetostriction and magnetostatic energies. Note that the magnetization in the boundary triangles in not divergence free.}
\label{fig: zig-zag_5}
\end{center}
\end{figure}
\noindent
It remains to estimate the magnetostatic energy of our construction in Figure \ref{fig: zig-zag_5}. We make use of 
\begin{lemma} \label{prop:demag} Let $\bbm \in L^2(G;\R^2)$ be a magnetization pattern, and let $\mathbf{h}_{\bbm} \in L^2(\R^2,\R^2)$ denote the corresponding induced magnetic field that satisfies Maxwell's equations of magnetostatics (\ref{maxwella}, \ref{maxwellb}) in the sense of distributions. Then
\item 
\begin{align*}
\int_{\R^2} |\mathbf{h}_\bbm|^2 \,d {\bf{x}} \leqslant \int_G |\bbm|^2\,d{\bf{x}}. 
\end{align*}
In fact,
\begin{align*}
\int_{\R^2} |\mathbf{h}_\bbm|^2 \,d {\bf{x}} = \min_{ \mathbf{n} \in \mathcal{B} } \int_{\R^2} |\mathbf{n}|^2 \,d{\bf{x}},
\end{align*}
with 
\begin{align*}
    \mathcal{B} := \{\mathbf{n} \in L^2(\R^2; \R^2) : \int_{\R^2} ( \mathbf{n} + \bbm) \cdot \nabla \psi \, d {\bf{x}} = 0 \mbox{ for every } \psi \in H^1(\R^2) \}.
\end{align*}
\end{lemma}
\begin{proof}
Let $\mathbf{h}_{\bbm} = -\nabla \chi$ where $\chi \in H^1(\R^2,\R^2)$. The short proof of this lemma is, for any $\mathbf{n} \in \mathcal{B}, $ we have  
\begin{align*}
    \int_{\R^2} |\mathbf{n}|^2 \, d {\bf{x}} &= \int_{\R^2} |\mathbf{n} - \mathbf{h}_{\bbm}|^2 + |\mathbf{h}_\bbm|^2 + 2 \langle (\mathbf{n} - \mathbf{h}_\bbm ), \mathbf{h}_\bbm \rangle \, d {\bf{x}} \\
    & \geqslant \int_{\R^2}|\mathbf{h}_\bbm|^2 - 2 \int_{\R^2} \langle \mathbf{n} + \bbm, \nabla \chi \rangle \\
    & = \int_{\R^2}|\mathbf{h}_\bbm|^2 \, d {\bf{x}},
\end{align*}
where in the second-to-last line, we have used Maxwell's equations and the fact that $\mathbf{n} \in \mathcal{B}.$ We note that we have equality if and only if $\mathbf{n} = \mathbf{h}_{\bbm}.$
\end{proof}

\noindent Observe that the above lemma does \textit{not} require that the test vector field $\mathbf{n}$ has support equal to that of $\bbm$; in fact, the vector field $\mathbf{n}$ is not even required to be $S^1$-valued in the domain.  We choose the test function $\mathbf{n}$ as follows: $\mathbf{n} = -\mathbf{m}$ on the boundary triangles and zero elsewhere, so that $\mathbf{n}$ is supported on the boundary triangles. Since $\dive \mathbf{n} = -\dive  \bbm$ in the sense of distributions on $\R^2$, by Lemma \ref{prop:demag}, we have
\begin{align*}
K_d \int_{\R^2} |\mathbf{h}_\bbm|^2 \,d {\bf{x}} \leq K_d \int_{\R^2} |\mathbf{n}|^2 \,d{\bf{x}}= K_d \int_{\text{bdry. triangles}} |\mathbf{n}|^2 \,d{\bf{x}} \sim  K_d \times \frac{L^{2}}{lk}
\end{align*} 
\noindent Hence, arguing as in \cite{JNLS2018}, the total sharp interface micromagnetic energy $\mathbb{F}_\#(\bbm)$ has three contributions, estimated by 
\begin{align}
\mathbb{F}_\#(\bbm) \lesssim \underbrace{\gamma L k}_\text{$180^{\circ}$ degree wall energy}+\underbrace{\gamma L  l}_\text{$90^{\circ}$ degree wall energy} + \underbrace{\bigg\{ c_{44}\lambda_{111}^2 + K_d  \bigg\} \times\frac{L^2}{lk}.}_\text{magnetostriction and magnetostatic energy}
\label{zigzagenergy1}
\end{align} 
\noindent Optimizing equation (\ref{zigzagenergy1}) with respect to  {\it{l}} and {\it{k}} we obtain our upper bound for a cubic ferromagnet with large and comparable magnetostriction and magnetostatic energies, both of which are  dominated by the anisotropy energy. Returning to our non-dimensional units and by standard facts about the Modica-Mortola $\eta \to 0$ asymptotics, the upper bound stated in Theorem \ref{mt} follows. \par

\section{Proof of the lower bound} \label{sec:lbnd}
The goal of this section is to prove Theorem \ref{mt}. We explain a reduction first. For any $\eta > 0,$ let $v_\eta \in H^1(G;\R^2)$ denote a minimizer of the energy $\mathcal{F}_\eta.$ The existence of such a minimizer follows by an easy application of the direct method in the Calculus of variations. By the upper bound construction, $\lim_{\eta \to 0} F_\eta(v_\eta) \lesssim \mu^{2/3}$ whenever $\mu < 1.$  By \cite{sternberg86}, after passing to a sub-sequence that is not denoted, $v_\eta \to \bbm$ strongly in $L^2(G;\R^2)$ where $|\bbm| = 1$ almost everywhere in $G$ and $\bbm \in BV(G;\R^2).$ Furthermore, thanks to the bound on the magnetocrystalline anisotropy, we in fact have $\bbm \in \left\{ \left( \pm \frac{1}{\sqrt{2}}, \pm \frac{1}{\sqrt{2}} \right)\right\} =\mathbb{K}$ almost everywhere in $G.$ In short, these entail that $\bbm \in \mathcal{M}.$   \par 
Let $\bbu_\eta$ denote the displacement associated to $v_\eta$, guaranteed by Theorem \ref{directmethod}. It follows then that $\bbu_\eta {\rightharpoonup} \bbu$  in $H^1(G;\R^2)$ where $\bbu$ is the displacement associated to $\bbm$; furthermore, we have 
\begin{align}
    \int_G \|\epsilon (\bbu) - \epsilon_0(\bbm)\|^2 \, d {\bf{x}} \leqslant \lim_{\eta \to 0} \int_G\|\epsilon(\bbu_\eta)  - \epsilon_0(v_\eta)\|^2 \, d {\bf{x}}. 
\end{align}\\
Finally, by the Modica-Mortola inequality \cite{sternberg86}, and using the fact that $\bbm \in \left\{ \left( \pm \frac{1}{\sqrt{2}}, \pm \frac{1}{\sqrt{2}} \right)\right\}$ almost everywhere in $G,$ we find that 
\begin{align}
    \label{modica}
    \mu \int_G |\nabla \bbm| \lesssim \liminf_{\eta \to 0} \int_G \mu \, \eta |\nabla v_\eta|^2 + \frac{\mu}{\eta} \varphi(v_\eta) \leqslant \liminf_{\eta \to 0} F_\eta(v_\eta;G). 
\end{align}
It is for these reasons that we state the theorem in terms of the asymptotic energy $\mathcal{F}_0.$

\begin{proof}[Proof of Theorem \ref{mt}]
The proof of the upper bound follows easily from the discussion in Section \ref{sec:upbnd}, and expressing the construction there in the non-dimensional units. It remains to prove the lower bound. The proof of the lower bound theorem proceeds in several steps. \\
For the proof of the lower bound, let $\bbm \in \mathcal{M}_0.$
For the convenience of the reader, we summarize the structure of the proof:
\begin{itemize}
    \item In Step 1, we simplify the magnetostriction energy in Fourier space. The key idea is to write this energy in terms of a Fourier multiplier acting on the oscillatory function $m_1 m_2$ which is $\pm \frac{1}{2}-$valued on $G.$ Note that the quantity $m_1 m_2$ corresponds to the off-diagonal terms in the preferred strain matrix $\epsilon_0(\bbm)$, and changes sign on $G,$ while the diagonal terms of $\epsilon_0(\bbm)$ are constant and equal to $\frac{1}{2}.$
{    \item In Step 2, we initiate a contradiction argument. If the minimum energy of $\mathcal{F}_0$ scales much smaller than $\mu^{2/3}$ for small $\mu,$ we use an interpolation inequality to obtain compactness in a Besov space of functions with one-third of a derivative. The resulting strong convergence of our sequence of test magnetizations yields has sufficient regularity to prove that the limiting magnetization has zero entropy production \cite{GhirLam}.  We derive the desired contradiction by our appeal to the regularity result of \cite{JOP}. }
  
\end{itemize}
\textbf{Step 1.} The goal of this step is to simplify the magnetostriction energy. Let $\bbm \in \mathcal{M}_0.$ We note that $V = \epsilon_0(\bbm)$ is in $L^p(G;\R^{2\times 2})$ for each $p \in [1,\infty]$ and is compactly supported. Let $\bbu \in H^1(G;\R^2)$ with $\int_G \bbu \,d{\bf{x}} = 0,$ and note that this mean-zero condition is merely a choice: in the following estimates it is convenient to arrange 
\begin{align}
    \label{avg}
    \frac{\langle \nabla \bbu \rangle + \langle \nabla \bbu \rangle^T}{2} = \langle V \rangle . 
\end{align}
 Taking into account Lemma \ref{lemfourrep}-Equation \eqref{Fourierrep}, and \eqref{avg}, 
\begin{align} \label{fourbeforesimplifying}
   \int_G \|\epsilon(\bbu) - V\|^2 \, d {\bf{x}} = \sum_{{\bf{k}} \in \mathbb{Z}^2, {\bf{k}} \neq 0} \frac{1}{|{\bf{k}}|^4}\left| |{\bf{k}}|^4 \|\widehat{V}({\bf{k}})\|^2 - 2|{\bf{k}}|^2 |\widehat{V}({\bf{k}}){\bf{k}}|^2 + |{\bf{k}} \cdot \widehat{V}({\bf{k}}){\bf{k}}|^2  \right|^2,
\end{align}
with $\widehat{V}({\bf{k}})$ being defined as in Lemma \ref{lemfourrep}. Since $m_1^2 = m_2^2 = \frac{1}{2}$ in $G,$ it follows that for ${\bf{k}} \in \mathbb{Z}^2 \backslash \{0\},$ we have $\widehat{m_1^2}({\bf{k}}) = \widehat{m_2^2}({\bf{k}}) = 0.$
Towards using \eqref{fourbeforesimplifying},  for ${\bf{k}} \neq {0},$ the matrix $\widehat{V}({\bf{k}})$ takes the form 
\begin{align} \label{structure}
    \widehat{V}({\bf{k}}) = \frac{3}{2}\begin{pmatrix}
    0 & b_k\\
    b_k & 0
    \end{pmatrix}
\end{align}
with
\begin{align}
    \hspace{1cm} b_k = \widehat{m_1 m_2}({\bf{k}}). 
\end{align}
Indeed, plugging in \eqref{structure} into \eqref{fourbeforesimplifying}, we find that for each ${\bf{k}}$, since 
\begin{align*}
    |{\bf{k}}|^4 \|\widehat{V}({\bf{k}})\|^2 - 2 |\mathbf{k}|^2 |\widehat{V}({\bf{k}}){\bf{k}}|^2 = 0,
\end{align*}
one has 
\begin{align} \label{magstricfourier}
   \int_G \|\epsilon(\bbu) - \epsilon_0(\bbm)\|^2 \, d {\bf{x}} = \sum_{{\bf{k}} \in \mathbb{Z}^2, {\bf{k}} \neq 0 } \frac{1}{|{\bf{k}}|^4}\left|2b_kk_1 k_2 \right|^2
\end{align}
 For brevity we set $g:= m_1 m_2.$ Working in the periodic setting and the method of images introduced in the proof of Lemma \ref{lemfourrep}, we note that the magnetostriction energy associated to the magnetization $\bbm$ is given by 
\begin{align}
 \int_G \|\epsilon(\bbu) - \epsilon_0(\bbm)\|^2 \, d {\bf{x}} = \Big\|\frac{\partial^2 g}{\partial x \partial y}\Big\|_{H^{-2}}^2.
\end{align}
\textbf{Step 2:} Suppose by way of contradiction that there exist a sequence $\mu_j \to 0$ and $\bbm_j \in \mathcal{M}_0$ with 
\begin{align} \label{e.toolittle}
\mathcal{F}_0(\bbm_j) \leqslant \beta_j \mu_j^{2/3},
\end{align}
for some sequence of positive numbers $\beta_j \to 0^+$ as $j \to \infty.$ 
We would like to obtain compactness of the $\bbm_j.$ We first work just with the exchange and magnetostriction energies: these being local, we use the periodic extensions introduced in Lemma \ref{lemfourrep} to note that, defining $g_j := m_1^j m_2^j,$ we have from Step 1 and \eqref{e.toolittle}
\begin{align}
&\Big\| \frac{\partial^2 g_j}{\partial x\partial y }\Big\|_{H^{-2}} \leqslant \beta_j\mu_j^{2/3} \label{e.compactness 1}\\
\label{e.compactness 2}
&\int_G |\nabla g_j| \leqslant \beta_j \mu_j^{-1/3}. 
\end{align}
Since $\bbm_j \in \mathcal{M}_0$ , we have that each $g_j$ has Fourier series supported in the set $\{\mathbf{k} = (k_1,k_2): k_1 k_2 \neq 0\}.$ On this set, $\frac{|k_1k_2|}{|k|^2} \gtrsim \frac{1}{|k|}.$ \\
Since for each $j,$ $g_j$ is supported on the set $\{\mathbf{k} = (k_1,k_2): k_1 k_2 \neq 0\},$ it follows that 
\begin{align}
\|g_j\|_{\mathring{H}^{-1}}^2 = \sum_{k \neq 0} \Big(\frac{1}{|\mathbf{k}|^2}|\hat{g}_j(\mathbf{k})|^2 \Big) \lesssim \sum_k \Big( \frac{|k_1 k_2|^2}{|k|^4}|\hat{g}(\mathbf{k})|^2   \Big)^2 \leqslant \beta_j O(\mu_j^{2/3}). 
\end{align}
In the following we use estimates from and related to Besov spaces on the torus; the reader is referred to \cite[Section 3.5]{frenchbook}. Now, we note that, by definition, $\mathring{H}^{-1} = \mathring{B}^{-1}_{2,2}.$ By complex interpolation, we note that for any $\theta \in (0,1),$ defining
\begin{equation}
\begin{aligned}
s := \theta(1) + (1-\theta)(-1), \\
\frac{1}{p} := \frac{\theta}{1} + \frac{1- \theta}{2}, \\
\frac{1}{q} := \frac{\theta }{\infty} + \frac{1- \theta}{2},
\end{aligned}
\end{equation}
we have the inequality 
\begin{align} \label{e.interp}
\|g_j\|_{\mathring{B}^s_{p,q}} \lesssim \|g_j \|_{\mathring{B}^1_{1,\infty}}^\theta \|g_j\|_{\mathring{B}^{-1}_{2,2}}^{1-\theta}. 
\end{align}
Towards getting the desired compactness,  we note that $\|g_j\|_{\mathring{B}^1_{1,\infty}} \lesssim \|\nabla g_j\|_{L^1} \ll \mu_j^{-1/3}$ and invoke \eqref{e.interp} with the choice $\theta = \frac{2}{3}.$ This entails that 
\begin{align}
\|g_j \|_{\mathring{B}^{1/3}_{6/5, 6} }\leqslant \beta_j. 
\end{align}
By the compactness of the embedding into the space $\mathring{B}^{1/3}_{6/5,6} \subset L^1,$ it follows that $g_j $ converges strongly to zero in $L^1$ at rate $\beta_j.$ \\
\textbf{Step 3:} The analysis of Step 2 entails that $g_j = m_1^jm_2^j$ is Cauchy in $L^1.$ Since $\bbm_j(\mathbf{x}) \in \mathbb{K}$ for almost every $\mathbf{x} \in \mathbb{K},$ it follows that $(m_1^j)^2 + (m_2^j)^2 \pm 2 m_1^j m_2^j$ is Cauchy in $L^1,$ or equivalently, $(m_1^j \pm m_2^j)$ is Cauchy in $L^2.$ This further entails that $m_1^j, m_2^j$ are Cauchy in $L^2$, and hence are strongly convergent to a limit $\bbm = (m_1,m_2).$ By passing to a subsequence, we obtain almost everywhere convergence, and hence that $\bbm \in \mathbb{K}$ almost everywhere. 

Identifying these periodic extensions with functions on $G,$ we find that $\bbm_j \to \bbm$ strongly in $L^2(G).$ On the other hand, this further entails that $\dive \, \bbm_j \to \dive \, \bbm $ strongly in $H^{-1}(\R^2).$ But since $\dive \, \bbm_j \to 0$ in $H^{-1};$ which implies that $\dive \, \bbm = 0$ in $H^{-1}(\R^2)$. Towards conclude the argument, we will invoke a regularity result from \cite{JOP}. To this end, we will show that $\bbm$ verifies a kinetic formulation of the eikonal equation. 
{
\par \textbf{Step 4:} We recall from Step 3 that $\bbm \in B^{1/3}_{6/5,6}(G)$ with $\bbm \in \mathbb{K}$ a.e. and $\dive \, \bbm = 0$ in the sense of distributions in $\R^2.$   Now we note that since $\bbm$ has finite range, for any $t \in (0,1)$ and any $z $ with $|z| \leqslant t$ we have 
\begin{align*}
\|\Delta_z \bbm\|_{L^{6/5}} \sim \|\Delta_z \bbm\|_{L^3(G)}.
\end{align*}
In particular, it follows that 
\begin{align*}
[\bbm]_{B^{1/3}_{3,6}}  < \infty.
\end{align*}
As $|\bbm | = 1$ and $|G| < \infty$ it follows that $\bbm \in B^{1/3}_{3,6}.$ Appealing to \cite[Section 4.1]{GhirLam} then, $\nabla\cdot \Phi(\bbm) = 0$ for all entropies $\Phi,$ cf. \cite[Definition 2.1]{dkmocompactness}. We remind the reader that a function $\Phi \in C^\infty_0(\R^2)$ is said to be an entropy if 
\begin{align*}
z \cdot D\,\Phi(z)z^\perp = 0, \quad \mbox{ for all } z \in \R^2, \quad \Phi(0) = 0, \quad D\, \Phi(0) = 0.
\end{align*}
Appealing then to \cite[Theorem 1.3]{JOP} it follows that $\bbm$ admits a  Lipschitz continuous  representative on any convex subset $\omega \Subset G,$ or it is a vortex in $\omega$. As $\bbm$ is finite range it can not be a vortex on any convex subset of $G$ and hence it must be constant. This is a contradiction since $\nabla \cdot \bbm = 0$ in the sense of distributions on $\R^2.$ 
}
\end{proof}

{\subsection*{Discussion.} We conclude the paper highlighting some features and limitations of our lower bound proof. 
\begin{enumerate}
\item We made crucial use of the concept of entropies that were introduced in \cite{dkmocompactness}. They were primarily developed as a tool to prove \textit{strong} compactness in $L^p$ spaces by exploiting compensation effects arising from an \textit{asymptotically increasing} penalty to the divergence of the magnetization in $H^{-1},$ and an \textit{asymptotically fading} penalty to the exchange energy. On the other hand, as suggested by the physics, the strength of the magnetostriction and magnetostatic coefficient are asymptotically \textit{order one}, i.e., comparable, while the exchange coefficient is still fading. Both these terms are nonlocal in nature and prefer oscillations, but in a sense, are in competition with each other. Given the easy axes, the demagnetization energy prefers a simple Landau state, which is very expensive for magnetostriction. On the other hand, magnetizations that are cheap for magnetostriction, such as constants or those that only use a pair of antipodal magnetizations, are very expensive for the demagnetization energy.  
\item For this reason above, we have not been able to use entropies to obtain compactness as in \cite{dkmocompactness}. Instead, we use the new magnetostriction term to buy us compactness. But for this, we must make assumptions on certain degenerate Fourier modes; it is this obstruction that limits us to the class $\mathcal{M}_0$. We believe that this is a technical restriction, and hope to pursue it elsewhere. Instead of using entropies for compactness, however, we use it for \textit{regularity}. Morally, our contradiction argument hinges on the magnetostatic energy whose vanishing requires that the limiting magnetization be tangent to $\partial \Omega.$ But even making sense of this requires a strong notion of trace along $\partial \Omega,$ which a generic $L^p$ function does not have. It is here that the specific regularity we prove, and the deep result of \cite{GhirLam} help us. This application bears analogy with scalar conservation laws: a result of A. Vasseur demonstrates the existence of strong $L^1$ traces for solutions of conservation laws with finite entropy production \cite{vasseur}. 
\item We believe that the assumption of membership to class $\mathcal{M}_0$ is not too restrictive: it respects the natural symmetries of the problem, with respect to the wells $\mathbb{K}.$ More importantly, our constructions satisfy the assumptions in this class. Roughly, it says that the construction ``makes use of all four of the wells'' for the magnetization equally.
\item That said, it would be very desirable to remove this restriction: for instance, a modification of the Privorotskii construction which is based on branching, is found in Iron. This construction was studied by Lifshitz as early as 1945 \cite{lifshitz1945magnetic}. This construction achieves the same scaling as our optimal one since it has zero magnetostriction by making use of only one pair of easy axes. For the same reason, it doesn't satisfy the assumptions for membership in $\mathcal{M}_0.$ 
\end{enumerate}
}
\section*{Acknowledgements} We would like to thank Robert V. Kohn for several useful comments on an earlier draft of this paper, and an anonymous referee for catching an error in an earlier version. We thank Felix Otto for pointing out a small error in \cite[Figure 2 a]{JNLS2018} of our previous paper, where the middle zig-zag lines were inverted. The correct figure is Figure \ref{fig: zig-zag_2}, making magnetization is divergence free. \par 
\noindent  
The research of R.V was partially supported by the Center for Nonlinear Analysis at Carnegie Mellon University, by an AMS-Simons travel award, and by the National
Science Foundation Grant No. DMS-1411646. The work of RDJ was supported by NSF (DMREF-1629026), and it also benefitted from the support of ONR (N00014-18-1-2766), 
the MURI Program (FA9550-12-1-0458, FA9550-16-1-0566), the RDF Fund of IonE, the Norwegian Centennial Chair Program
and the hospitality and support of the Isaac Newton Institute (EPSRC grant EP/R014604/1).
\renewcommand{\refname}{\normalsize \text{REFERENCES}}
\footnotesize
\bibliographystyle{plain}
\bibliography{references}

\begin{thebibliography}{10}

\bibitem{bbh}
Fabrice Bethuel, Ha\"{i}m Brezis, and Fr\'{e}d\'{e}ric H\'{e}lein.
\newblock {\em Ginzburg-{L}andau vortices}, volume~13 of {\em Progress in
  Nonlinear Differential Equations and their Applications}.
\newblock Birkh\"{a}user Boston, Inc., Boston, MA, 1994.

\bibitem{choksi1998bounds}
Rustum Choksi and Robert~V Kohn.
\newblock Bounds on the micromagnetic energy of a uniaxial ferromagnet.
\newblock {\em Communications on pure and applied mathematics}, 51(3):259--289,
  1998.

\bibitem{choksi1999domain}
Rustum Choksi, Robert~V Kohn, and Felix Otto.
\newblock Domain branching in uniaxial ferromagnets: a scaling law for the
  minimum energy.
\newblock {\em Communications in mathematical physics}, 201(1):61--79, 1999.

\bibitem{chopra2015non}
Harsh~Deep Chopra and Manfred Wuttig.
\newblock Non-joulian magnetostriction.
\newblock {\em Nature}, 521(7552):340--343, 2015.

\bibitem{JNLS2018}
Vivekanand Dabade, Raghavendra Venkatraman, and Richard~D James.
\newblock Micromagnetics of galfenol.
\newblock {\em Journal of Nonlinear Science}, pages 1--46, 2018.

\bibitem{desimone2002constrained}
Antonio DeSimone and Richard~D James.
\newblock A constrained theory of magnetoelasticity.
\newblock {\em Journal of the Mechanics and Physics of Solids}, 50(2):283--320,
  2002.

\bibitem{desimone2006recent}
Antonio DeSimone, Robert~V Kohn, Stefan M{\"u}ller, Felix Otto, et~al.
\newblock Recent analytical developments in micromagnetics.
\newblock {\em The science of hysteresis}, 2(4):269--381, 2006.

\bibitem{desimone2001two}
Antonio DeSimone, Robert~V Kohn, Stefan M{\"u}ller, Felix Otto, and Rudolf
  Sch{\"a}fer.
\newblock Two--dimensional modelling of soft ferromagnetic films.
\newblock In {\em Proceedings of the Royal Society of London A: Mathematical,
  Physical and Engineering Sciences}, volume 457, pages 2983--2991. The Royal
  Society, 2001.

\bibitem{dkmocompactness}
Antonio DeSimone, Stefan M\"{u}ller, Robert~V. Kohn, and Felix Otto.
\newblock A compactness result in the gradient theory of phase transitions.
\newblock {\em Proc. Roy. Soc. Edinburgh Sect. A}, 131(4):833--844, 2001.

\bibitem{GhirLam}
Francesco Ghiraldin and Xavier Lamy.
\newblock Optimal {B}esov differentiability for entropy solutions of the
  eikonal equation.
\newblock {\em Comm. Pure Appl. Math.}, 73(2):317--349, 2020.

\bibitem{fengbo}
Feng~Bo Hang and Fang~Hua Lin.
\newblock Static theory for planar ferromagnets and antiferromagnets.
\newblock {\em Acta Math. Sin. (Engl. Ser.)}, 17(4):541--580, 2001.

\bibitem{hubert2008magnetic}
Alex Hubert and Rudolf Sch{\"a}fer.
\newblock {\em Magnetic domains: the analysis of magnetic microstructures}.
\newblock Springer Science \& Business Media, 2008.

\bibitem{JOP}
Pierre-Emmanuel Jabin, Felix Otto, and Beno{\^\i}t Perthame.
\newblock Line-energy ginzburg-landau models: zero-energy states.
\newblock {\em Annali della Scuola Normale Superiore di Pisa-Classe di
  Scienze}, 1(1):187--202, 2002.

\bibitem{james1998magnetostriction}
Richard~D James and Manfred Wuttig.
\newblock Magnetostriction of martensite.
\newblock {\em Philosophical Magazine A}, 77(5):1273--1299, 1998.

\bibitem{knupfer2013nucleation}
Hans Kn{\"u}pfer, Robert~V Kohn, and Felix Otto.
\newblock Nucleation barriers for the cubic-to-tetragonal phase transformation.
\newblock {\em Communications on pure and applied mathematics}, 66(6):867--904,
  2013.

\bibitem{lifshitz1945magnetic}
E~Lifshitz.
\newblock On the magnetic structure of iron.
\newblock {\em ZHURNAL EKSPERIMENTALNOI I TEORETICHESKOI FIZIKI},
  15(3):97--107, 1945.

\bibitem{Nitsche}
Joachim~A Nitsche.
\newblock On korn's second inequality.
\newblock {\em RAIRO. Analyse num{\'e}rique}, 15(3):237--248, 1981.

\bibitem{frenchbook}
Hans-J\"{u}rgen Schmeisser and Hans Triebel.
\newblock {\em Topics in {F}ourier analysis and function spaces}.
\newblock A Wiley-Interscience Publication. John Wiley \& Sons, Ltd.,
  Chichester, 1987.

\bibitem{sternberg86}
Peter Sternberg.
\newblock The effect of a singular perturbation on nonconvex variational
  problems.
\newblock {\em Archive for Rational Mechanics and Analysis}, 101(3):209--260,
  1988.

\bibitem{vasseur}
Alexis Vasseur.
\newblock Strong traces for solutions of multidimensional scalar conservation
  laws.
\newblock {\em Arch. Ration. Mech. Anal.}, 160(3):181--193, 2001.

\bibitem{zhang2005phase}
JX~Zhang and LQ~Chen.
\newblock Phase-field microelasticity theory and micromagnetic simulations of
  domain structures in giant magnetostrictive materials.
\newblock {\em Acta Materialia}, 53(9):2845--2855, 2005.

\end{thebibliography}
\end{document}